\documentclass[reqno, 12pt]{amsart}
\usepackage{amssymb}
\usepackage{amsmath}
\usepackage[mathscr]{euscript}
\usepackage{mathtools}
\usepackage[top=2cm,bottom=3cm,left=2.5cm,right=2.5cm]{geometry}

\newtheorem{theorem}{Theorem}

\newtheorem{lemma}{Lemma}
\newtheorem{definition}{Definition}

\theoremstyle{remark}
\newtheorem{example}{Example}
\newtheorem{remark}{Remark}

\DeclareMathOperator{\Id}{Id}

\DeclareMathOperator{\Ker}{Ker}
\DeclareMathOperator{\Ima}{Im}

\title{Hyers-Ulam stability for finite-dimensional nonautonomous dynamics
}

\author{ Davor Dragi\v cevi\'c}
\address[Davor Dragi\v cevi\'c]
{Faculty of Mathematics\\
    University of Rijeka\\
    Rijeka 51000, Croatia}
\email[D.~Dragi\v cevi\'c]{ddragicevic@math.uniri.hr}

\begin{document}
\maketitle

\begin{abstract}
The main purpose of this paper is to obtain necessary and sufficient conditions under which a nonautonomous, finite-dimensional and  two-sided dynamics generated by a sequence of matrices or a linear
ODE exhibits Hyers-Ulam stability. Specifically,
in the case of discrete time we consider a nonautonomous difference equation with possibly noninvertible coefficients, 
while in the case of continuous time we deal with a nonautonomous ordinary differential equation without any bounded growth assumptions. 

\vspace{3mm}


\vspace{2mm}

\end{abstract}

\section{Introduction}
In the recent years many works have been devoted to the investigation of Hyers-Ulam stability for various classes of differential and difference equations. Roughly speaking, we say that a given differential or difference equation exhibits Hyers-Ulam stability if in a vicinity of  its approximate solutions, we can construct an exact solution. The importance of this notion stems from the fact we are usually unable to explicitly solve a given difference or differential equation, and that any numerical scheme will only  result in an  approximate solution of the equation. Thus, it is important to know that close to an approximate solution, there exists a true   solution of our equation.

In the present paper, we consider two classes of nonautonomous dynamics. More precisely, we deal with a nonautonomous difference equation of the form
\begin{equation}\label{LDEE}
x_{n+1}=A_n x_n \quad n\in \mathbb Z,
\end{equation}
where $(A_n)_{n\in \mathbb Z}$ is a sequence of linear operators on $\mathbb R^d$, as well as ordinary differential equations
\begin{equation}\label{ODE}
x'=A(t)x \quad t\in \mathbb R,
\end{equation}
where $A(t)$ is a  linear operator on $\mathbb R^d$ for each $t\in \mathbb R$, and the map $t\mapsto A(t)$ is continuous.
In~\cite{BD}, Backes and Dragi\v cevi\' c proved that Eq~\eqref{LDEE} and Eq.~\eqref{ODE} are Hyers-Ulam stable provided that these equations admit an exponential trichotomy (in the sense of Elaydi and Hajek~\cite{EH}). Moreover, they established converse result implying that Hyers-Ulam stability of Eq.~\eqref{LDEE} and Eq~\eqref{ODE} yields the existence of exponential trichotomy for these equations provided that operators $A_n$ are invertible and that $\sup_{t\in \mathbb R}\|A(t)\|<+\infty$ (actually for the approach in~\cite{BD} to work, it is sufficient that Eq.~\eqref{ODE} admits bounded growth). For related results that deal with the case when maps $n\mapsto A_n$ and $t\mapsto A(t)$ are periodic, we refer to~\cite{AO, BBT, BLO, BOST, O} and references therein. For other relevant contributions to the Hyers-Ulam stability and shadowing of  nonautonomous dynamics, see~\cite{AOR, BD21, BDOP, BDT, D, DJP, P}. In particular, in~\cite{BD21, BDOP, BDT, DJP, P} sufficient conditions under which nonlinear difference (or differential) equations obtained by perturbing Eq.~\eqref{LDEE} (or Eq.~\eqref{ODE}) exhibit (conditional) Hyers-Ulam, Hyers-Ulam-Rassias, or general (parametrized) shadowing properties were discussed. These works contain no converse results guaranteeing appropriate hyperbolicity of Eq.~\eqref{LDEE} (or Eq.~\eqref{ODE}) under Hyers-Ulam stability.

The main objective of the present paper is to revisit the above mentioned converse results from~\cite{BD}, and to remove the assumption on invertibility of coefficients $A_n$ in Eq.~\eqref{LDEE} and of bounded growth related to Eq.~\eqref{ODE}. More precisely, we prove that the Hyers-Ulam stability of Eq.~\eqref{LDEE} is equivalent to exponential trichotomy under a relatively mild condition regarding the uniqueness of backward bounded solutions. In Example~\ref{EX1} we show that this condition cannot be omitted.  

On the other hand, in the case of continuous time we show that the Hyers-Ulam stability of Eq.~\eqref{ODE} is equivalent to the concept of \emph{summable trichotomy}. This notion is inspired by the previously mentioned concept of exponential trichotomy, as well as the notion of summable dichotomy which goes back to the work of Coppel~\cite{Coppel}. For different concepts of trichotomy and their role in the study of asymptotic behavior of nonautonomous dynamical systems we refer to~\cite{SS16, SS24} and references therein.
\section{Discrete time case}
\subsection{Preliminaries}
Throughout this paper, $\mathbb R^d$ will denote the $d$-dimensional Euclidean space equipped with some norm $|\cdot |$. By $\|\cdot \|$ we will denote the associated matrix norm on the space of all linear operators acting on $\mathbb R^d$.
Let $(A_n)_{n\in \mathbb Z}$ be a sequence of (not necessarily invertible) linear operators on $\mathbb R^d$.  We consider the associated linear difference equation given by 
\begin{equation}\label{lde}
x_{n+1}=A_nx_n, \quad n\in \mathbb Z.
\end{equation}
By $\mathcal A(m,n)$ we will denote the \emph{linear cocycle} associated with~\eqref{lde} which is given by
\[
\mathcal A(m, n)=\begin{cases}
A_{m-1} \cdots A_n & m>n; \\
\Id & m=n,
\end{cases}
\]
where $\Id$ denotes the identity operator on $\mathbb R^d$.
We recall the notion of Hyers-Ulam stability for~\eqref{lde}.
\begin{definition}\label{HU}
We say that Eq.~\eqref{lde} is \emph{Hyers-Ulam stable} if there exists $L>0$ with the property that for each $\delta >0$ and a sequence $(y_n)_{n\in \mathbb Z}\subset \mathbb R^d$ such that 
\begin{equation}\label{pse}
\sup_{n\in \mathbb Z}|y_{n+1}-A_n y_n| \le \delta, 
\end{equation}
there exists a solution $(x_n)_{n\in \mathbb Z}\subset \mathbb R^d$ of Eq.~\eqref{lde} such that
\begin{equation}\label{approx}
\sup_{n\in \mathbb Z}|x_n-y_n| \le L\delta.
\end{equation}
\end{definition}
We will also consider the stronger concept of Hyers-Ulam stability.
\begin{definition}\label{HUU}
We say that Eq.~\eqref{lde} is \emph{Hyers-Ulam stable with uniqueness} if there exists $L>0$ with the property that for each $\delta >0$ and a sequence $(y_n)_{n\in \mathbb Z}\subset \mathbb R^d$ such that~\eqref{pse} holds, there exists a unique solution $(x_n)_{n\in \mathbb Z}\subset \mathbb R^d$ of Eq.~\eqref{lde} satisfying~\eqref{approx}.
\end{definition}

\begin{remark}
Obviously, if~\eqref{lde} is Hyers-Ulam stable with uniqueness then it is Hyers-Ulam stable. 
\end{remark}

We also recall the notion of exponential dichotomy.
\begin{definition}
Let $J\in \{\mathbb Z, \mathbb Z^+, \mathbb Z^-\}$, where $\mathbb Z^+=\{n\in \mathbb Z: \ n\ge 0\}$ and $\mathbb Z^-=\{n\in \mathbb Z: \ n\le 0\}$. We say that Eq.~\eqref{lde} admits an \emph{exponential dichotomy} on $J$ if there exist a family of projections $P_n$, $n\in J$ and constants $D, \lambda >0$ such that the following conditions hold:
\begin{itemize}
\item for $n, n+1\in J$, 
\[
P_{n+1}A_n=A_n P_n, 
\]
 and $A_n\rvert_{\Ker P_n} \colon \Ker P_n \to \Ker P_{n+1}$ is invertible;
\item for $m, n\in J$ such that $m\ge n$, 
\[
\|\mathcal A(m,n)P_n\| \le De^{-\lambda (m-n)};
\]
\item for $m, n\in J$ such that $m\le n$,
\[
\|\mathcal A(m,n)(\Id-P_n)\| \le De^{-\lambda (n-m)},
\]
where 
\[
\mathcal A(m,n):=\left (\mathcal A(n,m)\rvert_{\Ker P_m} \right)^{-1}\colon \Ker P_n \to \Ker P_m, \quad m<n.
\]
\end{itemize}
\end{definition}

\begin{remark}
The following facts are well-known (see for example~\cite[Lemma 3.1]{BFP2}):
\begin{enumerate}
\item if Eq.~\eqref{lde} admits an exponential dichotomy on $\mathbb Z^+$ with respect to projections $P_n^+$, $n\in \mathbb Z^+$, then 
\[
\Ima P_n^+=\left \{v\in \mathbb R^d: \ \sup_{m\ge n}|\mathcal A(m,n)v|<+\infty \right \};
\]
\item if Eq.~\eqref{lde} admits an exponential dichotomy on $\mathbb Z^-$ with respect to projections $P_n^-$, $n\in \mathbb Z^-$, then $\Ker P_n^-$ consists of  $v\in \mathbb R^d$ with the property that there exists a sequence $(x_m)_{m\le n}\subset \mathbb R^d$ such that $x_n=v$, $x_m=A_{m-1}x_{m-1}$ for $m\le n$ and $\sup_{m\le n}|x_m|<+\infty$.
\end{enumerate}
\end{remark}

Finally, we recall the notion of exponential trichotomy originally introduced by Elaydi and Hajek~\cite{EH}.
\begin{definition}
We say that~\eqref{lde} admits an \emph{exponential trichotomy} if the following conditions hold:
\begin{itemize}
\item Eq.~\eqref{lde} admits an exponential dichotomy on $\mathbb Z^+$ with respect to projections $P_n^+$, $n\in \mathbb Z^+$;
\item Eq.~\eqref{lde} admits an exponential dichotomy on $\mathbb Z^-$ with respect to projections $P_n^-$, $n\in \mathbb Z^-$;
\item we have that 
\begin{equation}\label{pro}
P_0^-=P_0^-P_0^+=P_0^+P_0^-.
\end{equation}
\end{itemize}
\end{definition}

Our first result is the following:
\begin{theorem}\label{T}
Suppose that the only bounded sequence $(x_n)_{n\le 0} \subset \mathbb R^d$ such that $x_0=0$ and $x_{n+1}=A_n x_n$, $n\le -1$ is the zero-sequence, i.e. $x_n=0$ for $n\le 0$. Then, 
the following properties are equivalent:
\begin{enumerate}
\item[(a)] Eq.~\eqref{lde} is Hyers-Ulam stable;
\item[(b)] Eq.~\eqref{lde} admits an exponential trichotomy.
\end{enumerate}
\end{theorem}

\begin{proof}
$(a)\implies (b)$
Suppose that Eq.~\eqref{lde} is Hyers-Ulam stable and let $L>0$ be as in the Definition~\ref{HU}. 
\begin{lemma}\label{L1}
For $m\in \mathbb N$ and a sequence $(w_n)_{n\ge -m}\subset \mathbb R^d$ with $\sup_{n\ge -m}|w_n|<+\infty$, there exists a sequence $(z_n)_{n\in \mathbb Z}\subset \mathbb R^d$ such that 
\begin{equation}\label{zw}
z_{n+1}=A_n z_n +w_n \quad \text{for $n\ge -m$,}
\end{equation}
and 
\begin{equation}\label{l}
\sup_{n\in \mathbb Z}|z_n| \le L\sup_{n\ge -m}|w_n|.
\end{equation}
\end{lemma}

\begin{proof}[Proof of the lemma]
We define a  sequence $(y_n)_{n\in \mathbb Z}\subset \mathbb R^d$ by 
\[
y_n=\begin{cases}
0 & n\le -m; \\
A_{n-1}y_{n-1}+w_{n-1} & n>-m.
\end{cases}
\]
Observe that 
\[
y_n-A_{n-1}y_{n-1}=\begin{cases}
0 & n\le -m; \\
w_{n-1} & n>-m.
\end{cases}
\]
In particular, we have that
\[
\sup_{n\in \mathbb Z}|y_n-A_{n-1}y_{n-1}| =\sup_{n\ge -m}|w_n|.
\]
Since Eq.~\eqref{lde} is Hyers-Ulam stable, there exists a solution $(x_n)_{n\in \mathbb Z}\subset \mathbb R^d$ of Eq.~\eqref{lde} such that 
\begin{equation}\label{l1}
\sup_{n\in \mathbb Z}|x_n-y_n| \le L\sup_{n\in \mathbb Z}|y_{n+1}-A_n y_n|=L\sup_{n\ge -m}|w_n|.
\end{equation}
Set
\[
z_n:=y_n-x_n, \quad n\in \mathbb Z.
\]
By~\eqref{l1}, we have that~\eqref{l} holds. Moreover, since $(x_n)_{n\in \mathbb Z}$ is a solution of Eq.~\eqref{lde}, we have that
\[
z_{n+1}=y_{n+1}-x_{n+1}=A_n y_n+w_n-A_n x_n=A_n z_n+w_n \quad n\ge -m,
\]
which yields~\eqref{zw}.
The proof of the lemma is thus completed.
\end{proof}
\begin{lemma}\label{L2}
For each sequence $(w_n)_{n\in \mathbb Z}\subset \mathbb R^d$ such that $\sup_{n\in \mathbb Z}|w_n|<+\infty$, there exists a sequence $(z_n)_{n\in \mathbb Z}\subset \mathbb R^d$ satisfying 
\begin{equation}\label{adm}
z_{n+1}=A_n z_n+w_n \quad n\in \mathbb Z,
\end{equation}
and
\begin{equation}\label{bound}
\sup_{n\in \mathbb Z}|z_n| \le L\sup_{n\in \mathbb Z}|w_n|.
\end{equation}
\end{lemma}
\begin{proof}[Proof of the lemma]
For every $m\in \mathbb N$, we define a sequence $(w_n^m)_{n\ge -m}$ by $w_n^m=w_n$, $n\ge -m$. By Lemma~\ref{L1}, we conclude that there exists a sequence 
$(z_n^m)_{n\in \mathbb Z}\subset \mathbb R^d$ such that 
\begin{equation}\label{d1}
z_{n+1}^m=A_n z_n^m+w_n^m=A_n z_n^m+w_n \quad \text{for $n\ge -m$,}
\end{equation}
and 
\begin{equation}\label{d2}
\sup_{n\in \mathbb Z}|z_n^m| \le L\sup_{n\ge -m}|w_n^m|\le L\sup_{n\in \mathbb Z}|w_n|.
\end{equation}
It follows from~\eqref{d2} that for each $n\in \mathbb Z$, the sequence $(z_n^m)_m$ is a bounded sequence in $\mathbb R^d$, which therefore has a convergent subsequence. By applying the diagonal procedure, we can find a subsequence $(m_j)_{j\in \mathbb N}$ of $\mathbb N$ with the property that the sequence $(z_n^{m_j})_{j\in \mathbb N}$ converges for each $n\in \mathbb Z$. Let
\[
z_n:=\lim_{j\to \infty} z_n^{m_j}, \quad n\in \mathbb Z.
\]
Take now an arbitrary $n\in \mathbb Z$. Then, $n\ge -m_j$ for $j$ sufficiently large. 
By applying~\eqref{d1} for $m=m_j$ and passing to the limit when $j\to \infty$, we conclude that~\eqref{adm} holds. Similarly, \eqref{bound} follows from~\eqref{d2}. The proof of the lemma is completed.
\end{proof}
Set
\[
\mathcal S:=\left \{ v\in \mathbb R^d: \ \sup_{n\ge 0}|\mathcal A(n,0)v|<+\infty \right \}.
\]
Moreover, let $\mathcal U$ consist of all $v\in \mathbb R^d$ with the property that there exists a sequence $(x_n)_{n \in \mathbb Z^-}\subset \mathbb R^d$ such that $x_0=v$, $x_{n+1}=A_nx_n$ for $n\le -1$ and $\sup_{n\in \mathbb Z^-}|x_n|<+\infty$. Clearly, $\mathcal S$ and $\mathcal U$ are subspaces of $\mathbb R^d$.
\begin{lemma}
We have that 
\begin{equation}\label{split}
\mathbb R^d=\mathcal S+\mathcal U.
\end{equation}
\end{lemma}
\begin{proof}[Proof of the lemma]
Take $v\in \mathbb R^d$. We define a sequence $(w_n)_{n\in \mathbb Z}\subset \mathbb R^d$ by $w_{-1}=v$ and $w_n=0$ for $n\neq -1$. Clearly, $\sup_{n\in \mathbb Z}|w_n|=|v|<+\infty$. It follows from Lemma~\ref{L2} that there exists a sequence $(z_n)_{n\in \mathbb Z}\subset \mathbb R^d$ such that~\eqref{adm} holds and that 
$\sup_{n\in \mathbb Z}|z_n|<+\infty$. Observe that it follows from~\eqref{adm} that 
\[
z_0-A_{-1}z_{-1}=v \quad \text{and} \quad z_{n+1}=A_n z_n, \ n\neq -1.
\]
In particular, we have that  $z_n=\mathcal A(n, 0)z_0$ for $n\ge 0$. This implies that $z_0\in \mathcal S$. 
Moreover, we define $(x_n)_{n\in \mathbb Z^-}\subset \mathbb R^d$ by 
\[
x_n=\begin{cases}
z_n & n<0,\\
A_{-1}z_{-1} & n=0.
\end{cases}
\]
Since $x_{n+1}=A_n x_n$ for $n\le -1$ and $\sup_{n\in \mathbb Z^-}|x_n| <+\infty$, we have that $x_0=A_{-1}z_{-1}\in \mathcal U$. Hence,
\[
v=z_0-A_{-1}z_{-1}\in \mathcal S+\mathcal U.
\]
\end{proof}
Choose a subspace $Z\subset \mathcal U$ such that 
\begin{equation}\label{I}
\mathbb R^d=\mathcal S\oplus Z.
\end{equation}
\begin{lemma}\label{L3}
For each sequence $(w_n)_{n\in \mathbb Z^+}\subset \mathbb R^d$ such that $\sup_{n\in \mathbb Z^+}|w_n|<+\infty$, there exists a unique sequence $(x_n)_{n\in \mathbb Z^+}\subset \mathbb R^d$ such that $x_0\in Z$, $\sup_{n\ge 0}|x_n|<+\infty$ and 
\[
x_{n+1}=A_n x_n +w_n, \quad n\in \mathbb Z^+.
\]
\end{lemma}

\begin{proof}[Proof of the lemma]
Set $w_n=0$ for $n<0$. Since $\sup_{n\in \mathbb Z}|w_n|<+\infty$, it follows from Lemma~\ref{L2} that there exists a sequence $(z_n)_{n\in \mathbb Z}\subset \mathbb R^d$ such that~\eqref{adm} holds and that $\sup_{n\in \mathbb Z}|z_n|<+\infty$. By~\eqref{I}, there exist $v_1\in \mathcal S$ and $v_2\in Z$ such that $z_0=v_1+v_2$.  Let
\[
x_n=z_n-\mathcal A(n,0)v_1, \quad n\in \mathbb Z^+.
\]
It is clear that the sequence $(x_n)_{n\in \mathbb Z^+}$ has the desired properties. 

We now establish the uniqueness. Assume that $(\tilde x_n)_{n\in \mathbb Z^+}\subset \mathbb R^d$ is another sequence such that $\sup_{n\in \mathbb Z^+}|\tilde x_n|<+\infty$, $\tilde x_0\in Z$ and 
\[
\tilde x_{n+1}=A_n  \tilde x_n +w_n, \quad n\in \mathbb Z^+.
\]
Then, $x_n-\tilde x_n=\mathcal A(n, 0)(x_0-\tilde x_0)$ for $n\in \mathbb Z^+$. Hence, $x_0-\tilde x_0\in \mathcal S\cap Z$, and therefore $x_0=\tilde x_0$. We conclude that $x_n=\tilde x_n$ for $n\in \mathbb Z^+$.
\end{proof}
It follows from Lemma~\ref{L3} and~\cite[Corollary 1.10]{Henry} (see also~\cite[Corollary 4.4.]{SS}) that Eq.~\eqref{lde} admits an exponential dichotomy on $\mathbb Z^+$ with respect to projections $P_n^+$, $n\in \mathbb Z^+$ such that 
$\Ker P_0^+=Z$.

Take now $Z'\subset \mathcal S$ such that 
\[
\mathbb R^d=Z'\oplus \mathcal U.
\]
\begin{lemma}
For each sequence $(w_n)_{n\le -1}\subset \mathbb R^d$ such that $\sup_{n\le -1}|w_n|<+\infty$, there exists a unique sequence $(x_n)_{n\in \mathbb Z^-}\subset X$ such that $x_0\in Z'$, $\sup_{n\le 0}|x_n|<+\infty$ and 
\[
x_{n+1}=A_n x_n +w_n, \quad n\le -1.
\]
\end{lemma}
\begin{proof}[Proof of the lemma]
Set $w_n=0$ for $n\ge 0$. Since $\sup_{n\in \mathbb Z}|w_n|<+\infty$, it follows from Lemma~\ref{L2} that there exists a sequence $(z_n)_{n\in \mathbb Z}\subset \mathbb R^d$ such that~\eqref{adm} holds and that $\sup_{n\in \mathbb Z}|z_n|<+\infty$. Take $v_1\in Z'$ and $v_2\in \mathcal U$ such that $z_0=v_1+v_2$. Since $v_2\in \mathcal U$, there exists a sequence $(t_n)_{n\le 0}\subset \mathbb R^d$ such that $\sup_{n\le 0}|t_n|<+\infty$, $t_0=v_2$ and $t_{n+1}=A_nt_n$ for $n\le -1$. Set
\[
x_n:=z_n-t_n, \quad n\le 0.
\]
Then, the sequence $(x_n)_{n\le 0}$ has the desired properties. One can easily establish the uniqueness part.
\end{proof}
The previous lemma together with the assumption in the statement of the theorem  implies (see~\cite[Corollary 1.11]{Henry}) that Eq.~\eqref{lde} admits an exponential dichotomy on $\mathbb Z^-$ with respect to projections $P_n^-$, $n\in \mathbb Z^-$ such that 
$\Ima P_0^-=Z'$.  Observe that $\Ker P_0^+=Z\subset \mathcal U=\Ker P_0^-$ and $\Ima P_0^-=Z'\subset \mathcal S=\Ima P_0^+$. This easily implies that~\eqref{pro} holds.
We conclude that Eq.~\eqref{lde} admits an exponential trichotomy.

$(b)\implies (a)$
This implication is established in~\cite[Corollary 1]{BD}.
\end{proof}

\begin{remark}
The implication $(a)\implies (b)$ of Theorem~\ref{T} has been established in~\cite[Proposition 3]{BD} in the case when $A_n$ is an invertible operator for every $n\in \mathbb Z$. In this case it is not necessary to  make the assumption from Theorem~\ref{T}.
\end{remark}

\begin{example}\label{EX1}
Let us now give an explicit example (inspired by~\cite[Remark 2]{DP}) which illustrates that the assumption we made in the statement of Theorem~\ref{T} cannot be eliminated.  For this purpose,
take $d=1$ and consider a sequence $(A_n)_{n\in \mathbb Z}$ by 
\[
A_n=\begin{cases}
0 &n\ge -1 \\
2 &n<-1.
\end{cases}
\]
We claim that the associated equation Eq.~\eqref{lde} is Hyers-Ulam stable. To this end, let 
\[
\ell^\infty:=\left \{w=(w_n)_{n\in \mathbb Z}\subset \mathbb R: \ \|w\|_\infty:= \sup_{n\in \mathbb Z}|w_n|<+\infty \right \}.
\]
Then, $(\ell^\infty, \| \cdot \|_\infty)$ is a Banach space. We define $\Gamma \colon \ell^\infty \to \ell^\infty$ by 
\[
(\Gamma w)_n:=\begin{cases}
w_{n-1} &n\ge 0; \\
0 &n=-1; \\
-\sum_{i=1}^{-(n+1)}\frac{1}{2^i}w_{n+i-1} &n\le -2,
\end{cases}
\]
for $w=(w_n)_{n\in \mathbb Z}\in \ell^\infty$. Clearly, $\|\Gamma w\|_\infty \le \|w\|_\infty$. Moreover, it is easy to verify that 
\begin{equation}\label{j1}
(\Gamma w)_{n+1}=A_n (\Gamma w)_n+w_n, \quad n\in \mathbb Z.
\end{equation}
Take now a sequence $(y_n)_{n\in \mathbb Z}\subset \mathbb R$ such that   $\sup_{n\in \mathbb Z}|y_{n+1}-A_n y_n| \le \delta$ for some $\delta>0$. Set 
\begin{equation}\label{j2}
w_n:=y_{n+1}-A_n y_n, \quad n\in \mathbb Z.
\end{equation}
Furthermore, let $w:=(w_n)_{n\in \mathbb Z}\in \ell^\infty$ and 
\[
x_n:=y_n-(\Gamma w)_n, \quad n\in \mathbb Z.
\]
By~\eqref{j1} and~\eqref{j2}, we have that the sequence $(x_n)_{n\in \mathbb Z}$ is a solution of Eq.~\eqref{lde}. Moreover, 
\[
\sup_{n\in \mathbb Z}|x_n-y_n|=\| \Gamma w\|_\infty \le \|w\|_\infty \le \delta. 
\]
Thus, Eq.~\eqref{lde} is Hyers-Ulam stable. On the other hand, Eq.~\eqref{lde} does not admit exponential trichotomy as it does not admit exponential dichotomy on $\mathbb Z^-$.
\end{example}

We can also give a complete characterization of Hyers-Ulam stability with uniqueness. 
\begin{theorem}\label{99}
The following properties are equivalent:
\begin{enumerate}
\item Eq.~\eqref{lde} is Hyers-Ulam stable with uniqueness;
\item Eq.~\eqref{lde} admits an exponential dichotomy on $\mathbb Z$.
\end{enumerate}
\end{theorem}

\begin{proof}
Suppose that Eq.~\eqref{lde} is Hyers-Ulam stable with uniqueness. By Theorem~\ref{T}, we have that Eq.~\eqref{lde} admits an exponential trichotomy. Hence, Eq.~\eqref{lde} admits an exponential dichotomy on $\mathbb Z^+$ with projections $P_n^+$ and exponential dichotomy on $\mathbb Z^-$ with projections $P_n^-$. Furthermore, \eqref{pro} holds. Hence, $\Ima P_0^-\subset \Ima P_0^+$. Take now $v\in \Ima P_0^+$ and write it in the form $v=v_1+v_2$, where $v_1\in \Ima P_0^-$ and $v_2\in \Ker P_0^-$. Note that $v_2=v-v_1\in \Ima P_0^+$. Moreover, since $v_2\in \Ker P_0^-$, there exists a sequence $(x_n)_{n\in \mathbb Z^-}\subset \mathbb R^d$ such that $x_0=v_2$, $x_n=A_{n-1}x_{n-1}$ for $n\le 0$ and $\sup_{n\le 0}|x_n|<+\infty$. For $n\in \mathbb Z$, set
\[
\tilde x_n=\begin{cases}
\mathcal A(n,0)v_2 &n\ge 0; \\
x_n &n<0.
\end{cases}
\]
Clearly, $(\tilde x_n)_{n\in \mathbb Z}$ is a solution of Eq.~\eqref{lde} and $M:=\sup_{n\in \mathbb Z}|\tilde x_n|<+\infty$. Finally, we introduce a sequence $(y_n)_{n\in \mathbb Z}\subset \mathbb R^d$ by $y_n=\frac{L}{M}\tilde x_n$, $n\in \mathbb Z$. Note that the sequence $(y_n)_{n\in \mathbb Z}$ is a solution of Eq.~\eqref{lde}. In particular, \eqref{pse} holds with $\delta=1$. Hence, $(y_n)_{n\in \mathbb Z}$ is $L$-shadowed by itself and the constant solution $(0)_{n\in \mathbb Z}$ of Eq.~\eqref{lde}. By the uniqueness in Definition~\ref{HUU}, we conclude that $y_n=0$ for each $n\in \mathbb Z$. Hence, $\tilde x_0=v_2=0$. Consequently, $v=v_1\in \Ima P_0^-$ and we conclude that $\Ima P_0^-=\Ima P_0^+$. Since $\Ker P_0^+\subset \Ker P_0^-$, we have that $\Ker P_0^+=\Ker P_0^-$. This yields that $P_0^-=P_0^+$ and therefore Eq.~\eqref{lde} admits an exponential dichotomy.

Suppose that Eq.~\eqref{lde} admits an exponential dichotomy. By Theorem~\ref{T}, we have that Eq.~\eqref{lde} is Hyers-Ulam stable. It remains to establish the uniqueness. Suppose that for a sequence $(y_n)_{n\in \mathbb Z}\subset \mathbb R^d$ satisfying~\eqref{pse} with some $\delta>0$, there are two solutions $(x_n)_{n\in \mathbb Z}$ and $(\tilde x_n)_{n\in \mathbb Z}$ of Eq.~\eqref{lde} such that 
\[
\sup_{n\in \mathbb Z}|x_n-y_n| \le L\delta \quad \text{and} \quad \sup_{n\in \mathbb Z}|\tilde x_n-y_n| \le L\delta.
\]
Set $w_n:=x_n-\tilde x_n$, $n\in \mathbb Z$. Then, $(w_n)_{n\in \mathbb Z}$ is a solution of Eq.~\eqref{lde} such that $\sup_{n\in \mathbb Z}|w_n|<+\infty$. Since Eq.~\eqref{lde} admits an exponential dichotomy, we have that $w_n=0$ for each $n\in \mathbb Z$. We conclude that $x_n=\tilde x_n$ for $n\in \mathbb Z$. The proof of the theorem is completed.
\end{proof}

\section{Continuous time case}
Let $A\colon \mathbb R\to  \mathbb R^{d\times d}$ be a continuous map. We consider the associated linear differential equation given by
\begin{equation}\label{LDE}
x'=A(t)x, \quad t\in \mathbb R.
\end{equation}
By $T(t,s)$ we denote the evolution family associated to~\eqref{LDE}.
\begin{definition}\label{HU-c}
We say that Eq.~\eqref{LDE} is \emph{Hyers-Ulam stable} if there exists $L>0$ with the property that for each $\delta >0$ and a continuously differentiable function $y\colon \mathbb R \to \mathbb R^d$ satisfying
\begin{equation}\label{pse-c}
\sup_{t\in \mathbb R}|y'(t)-A(t)y(t)| \le \delta, 
\end{equation}
there exists a solution $x\colon \mathbb R\to \mathbb R^d$ of Eq.~\eqref{LDE} such that
\begin{equation}\label{approx-c}
\sup_{t\in \mathbb R}|x(t)-y(t)| \le L\delta.
\end{equation}
\end{definition}
Similarly, one can introduce the notion of \emph{Hyers-Ulam stability with uniqueness} for Eq.~\eqref{LDE} by requiring that $x$ in Definition~\ref{HU-c} is unique.
\begin{definition}
Let $J\in \{\mathbb R^+, \mathbb R^-, \mathbb R\}$, where $\mathbb R^+=[0, \infty)$ and $\mathbb R^-=(-\infty, 0]$. We say that Eq.~\eqref{LDE} admits a \emph{summable dichotomy} on $J$ if there exist a family of projections $P(t)$, $t\in J$ on $\mathbb R^d$ and a  constant $K>0$ such that the following holds:
\begin{itemize}
\item for $t, s\in J$, $T(t,s)P(s)=P(t)T(t,s)$;
\item for $t\in J$,
\begin{equation}\label{sum}
\int_{\inf J}^t\|T(t,s)P(s)\| \, ds+\int_t^{\sup J} \|T(t,s)Q(s)\|\, ds \le K,
\end{equation}
where $Q(s)=\Id-P(s)$.
\end{itemize}
\end{definition}

\begin{remark}\label{uu}
We recall that~\eqref{LDE}  admits an exponential dichotomy on $J$  if there exist constants $D, \lambda>0$ such that 
\begin{equation}\label{ed1}
\|T(t,s)P(s)\| \le De^{-\lambda (t-s)} \quad \text{for $t, s\in J$ with $t\ge s$,}
\end{equation}
and 
\begin{equation}\label{ed2}
\|T(t,s)Q(s)\| \le De^{-\lambda (s-t)} \quad \text{for $t, s\in J$ with $t\le s$.}
\end{equation}
Clearly, exponential dichotomy is a particular case of a summable dichotomy (as~\eqref{ed1} and~\eqref{ed2} imply~\eqref{sum}).

Take now a continuously differentiable function $\phi \colon [0, \infty) \to \mathbb R$ such that $0<\phi(t)\le 1$ for $t\ge 0$, $\int_0^\infty \left (\frac{1}{\phi(t)}-1\right )\, dt=1$ and $\lim_{n\to \infty}\frac{\phi(n)}{\phi(n-2^{-n})}=\infty$. We consider the equation
\begin{equation}\label{Copp}
x'=\left (\frac{\phi'(t)}{\phi(t)}-1\right )x.
\end{equation}
It is proved in~\cite[p.27]{Coppel2} that~\eqref{Copp} admits a summable dichotomy (with $P(t)=\Id$) but does not admit an exponential dichotomy.
\end{remark}

\begin{remark}\label{ct}
It is observed in~\cite[p.134]{Coppel} the following holds true:
\begin{itemize}
\item if Eq.~\eqref{LDE} admits a summable dichotomy on $\mathbb R^+$ with respect to projections $P(t)$, $t\ge 0$, then
\[
\Ima P(0)=\left \{ v\in \mathbb R^d: \ \sup_{t\ge 0} |T(t,0)v|<+\infty\right \}=:\mathcal S;
\]
Furthermore, $\Ker P(0)$ can be an arbitrary subspace of $\mathbb R^d$ complemented to $\mathcal S$;
\item if Eq.~\eqref{LDE} admits a summable dichotomy on $\mathbb R^-$ with respect to projections $P(t)$, $t\le 0$, then
\[
\Ker P(0)=\left \{ v\in \mathbb R^d: \ \sup_{t\le 0} |T(t,0)v|<+\infty\right \}=:\mathcal U;
\]
Furthermore, $\Ima P(0)$ can be an arbitrary subspace of $\mathbb R^d$ complemented to $\mathcal U$.
\end{itemize}
\end{remark}

\begin{definition}\label{ST}
We say that Eq.~\eqref{LDE} admits a \emph{summable trichotomy} if the following conditions hold:
\begin{itemize}
\item Eq.~\eqref{LDE} admits a summable dichotomy on $\mathbb R^+$ with respect to projections $P^+(t)$, $t\ge 0$;
\item Eq.~\eqref{LDE} admits a summable dichotomy on $\mathbb R^-$ with respect to projections $P^-(t)$, $t\le 0$;
\item we have that 
\begin{equation}\label{pro-c}
P^-(0)=P^-(0)P^+(0)=P^+(0)P^-(0).
\end{equation}
\end{itemize}
\end{definition}

\begin{remark}\label{ET}
We recall (see~\cite{EH}) that Eq.~\eqref{LDE} admits an \emph{exponential trichotomy} if:
\begin{itemize}
\item Eq.~\eqref{LDE} admits an exponential  dichotomy on $\mathbb R^+$ with respect to projections $P^+(t)$, $t\ge 0$;
\item Eq.~\eqref{LDE} admits an exponential dichotomy on $\mathbb R^-$ with respect to projections $P^-(t)$, $t\le 0$;
\item \eqref{pro-c} holds.
\end{itemize}
It follows from the discussion in Remark~\ref{uu} that the notion of exponential trichotomy is a particular case of the notion of a summable trichotomy.
\end{remark}

\begin{theorem}\label{yyy}
The following statements are equivalent:
\begin{enumerate}
\item[(a)] Eq.~\eqref{LDE} is Hyers-Ulam stable;
\item[(b)] Eq.~\eqref{LDE} admits  a summable trichotomy.
\end{enumerate}
\end{theorem}

\begin{proof}
$(a)\implies (b)$
Suppose that  Eq.~\eqref{LDE} is Hyers-Ulam stable. 
\begin{lemma}\label{09}
For each continuous function $z\colon \mathbb R\to \mathbb R^d$ such that $\sup_{t\in \mathbb R}|z(t)|<+\infty$, there exists  a continuously differentiable $x\colon \mathbb R\to \mathbb R^d$ satisfying
\begin{equation}\label{ADM}
x'(t)=A(t)x(t)+z(t) \quad t\in \mathbb R,
\end{equation}
and \begin{equation}\label{L}\sup_{t\in \mathbb R}|x(t)|\le L\sup_{t\in \mathbb R}|z(t)|.\end{equation}
\end{lemma}
\begin{proof}[Proof of the lemma]
Set $\|z\|_\infty:=\sup_{t\in \mathbb R}|z(t)|<+\infty$, and let $y\colon \mathbb R\to \mathbb R^d$ be an arbitrary solution of the equation
\[
y'(t)=A(t)y(t)+z(t), \quad t\in \mathbb R.
\]
Then, $\sup_{t\in \mathbb R}|y'(t)-A(t)y(t)|=\|z\|_\infty$. Hence, there exists a solution $\tilde x\colon \mathbb R\to \mathbb R^d$ of Eq.~\eqref{LDE} such that 
$\sup_{t\in \mathbb R}|y(t)-\tilde x(t)| \le L\|z\|_\infty$. Set $x(t):=y(t)-\tilde x(t)$, $t\in \mathbb R$. Clearly, $x$ has the desired properties. 
\end{proof}

\begin{lemma}\label{pp}
We have that 
\begin{equation}\label{split-c}
\mathbb R^d=\mathcal S\oplus \mathcal U.
\end{equation}
\end{lemma}
\begin{proof}[Proof of the lemma]
Take $v\in \mathbb R^d$. Choose a nonnegative continuous function $\phi \colon \mathbb R\to \mathbb R$ whose support is contained in $[0, 1]$ and such that $\int_0^1\phi(s)\, ds=1$. Let
\[
z(t)=\phi(t)T(t, 0)v, \quad t\in \mathbb R.
\]
Clearly, $z$ is continuous and $\sup_{t\in \mathbb R}|z(t)|<+\infty$. By Lemma~\ref{09}, there exists a continuously differentiable function $x\colon \mathbb R\to \mathbb R^d$ satisfying~\eqref{ADM} and such that $\sup_{t\in \mathbb R}|x(t)|<+\infty$. Observe that $x(t)=T(t,0)x(0)$ for $t\le 0$, and thus $x(0)\in \mathcal U$. Moreover, for $t\ge 1$ we have that 
\[
\begin{split}
x(t) &=T(t,0)x(0)+\int_0^t \phi(s) T(t,s)T(s, 0)v\, ds\\
&=T(t,0)x(0)+\left (\int_0^1\phi(s)\, ds \right )T(t,0)v \\
&=T(t,0)(x(0)+v).
\end{split}
\]
Hence, $x(0)+v\in \mathcal S$. We conclude that
\[
v=x(0)+v-x(0)\in \mathcal S+\mathcal U.
\]
\end{proof}
The following lemma is a direct consequence of Lemma~\ref{09}.
\begin{lemma}\label{10}
The following holds:
\begin{itemize}
\item for each continuous function $z\colon [0, \infty) \to \mathbb R^d$ such that $\sup_{t\ge 0}|z(t)|<+\infty$, there exists  a continuously differentiable $x\colon [0, \infty)\to \mathbb R^d$ satisfying
\begin{equation}\label{adm-pl}
x'(t)=A(t)x(t)+z(t) \quad t\ge 0,
\end{equation}
and $\sup_{t\ge 0}|x(t)| \le L\sup_{t\ge 0}|z(t)|$;
\item  for each continuous function $z\colon (-\infty, 0] \to \mathbb R^d$ such that $\sup_{t\le 0}|z(t)|<+\infty$, there exists  a continuously differentiable $x\colon (-\infty, 0]\to \mathbb R^d$ satisfying
\begin{equation}\label{adm-nl}
x'(t)=A(t)x(t)+z(t) \quad t\le 0,
\end{equation}
and $\sup_{t\le 0}|x(t)| \le L\sup_{t\le 0}|z(t)|$.
\end{itemize}

\end{lemma}
By~\eqref{split-c}, there exist subspaces $Z\subset \mathcal U$ and $Z'\subset \mathcal S$ such that 
\begin{equation}\label{RZZ}
\mathbb R^d=\mathcal S\oplus Z \quad \text{and} \quad \mathbb R^d=Z'\oplus \mathcal U.
\end{equation}
It follows from Lemma~\ref{10}, Remark~\ref{ct} and~\cite[Theorem 1, p.131]{Coppel} that Eq.~\eqref{LDE} admits a summable dichotomy on $\mathbb R^+$ with respect to projections $P^+(t)$, $t\ge 0$ and  a summable dichotomy on $\mathbb R^-$ with respect to projections $P^-(t)$, $t\le 0$. Moreover, $\Ker P^+(0)=Z$ and $\Ima P^-(0)=Z'$. This easily implies that~\eqref{pro-c} holds. Therefore, Eq.~\eqref{LDE}  admits a summable trichotomy.

$(b)\implies (a)$ Suppose that Eq.~\eqref{LDE} admits a summable trichotomy. By~\cite[Lemma 1, p.68]{Coppel}, there exists $N\ge 1$ such that $\|T(t,0)P^+(0)\| \le N$ for $t\ge 0$ and $\|T(t,0)P^-(0)\| \le N$ for $t\le 0$.
We need the following auxiliary  result.
\begin{lemma}
There exists $L>0$ with the property that for each continuous $z\colon \mathbb R\to \mathbb R^d$ such that $\sup_{t\in \mathbb R}|z(t)|<+\infty$, there exists  a continuously differentiable $x\colon \mathbb R\to \mathbb R^d$ sastifying~\eqref{ADM} and~\eqref{L}.
\end{lemma}
\begin{proof}[Proof of the lemma]
For $t\ge 0$, let 
\[
x_1(t)=\int_0^tT(t,s)P^+(s)z(s)\, ds-\int_t^\infty T(t,s)Q^+(s)z(s)\, ds,
\]
where $Q^+(s)=\Id-P^+(s)$. Then, \eqref{sum} implies that $\sup_{t\ge 0}|x_1(t)| \le K\sup_{t\ge 0}|z(t)|$. Observe that
\begin{equation}\label{ab}
x_1'(t)=A(t)x_1(t)+z(t) \ \text{for $t>0$} \ \text{and}  \ x_1'(0+)=A(0)x_1(0)+z(0),
\end{equation}
where $x_1'(0+)$ denotes the right-derivative of $x_1$ at $0$. On the other hand, observe that $x_1(0)\in \Ker P^+(0)\subset \Ker P^-(0)$. Hence, by setting $x_1(t):=T(t,0)x_1(0)$ for $t<0$ we can extend $x_1$ to a bounded function on $\mathbb R$. In fact, $|x_1(t)| \le N|x_1(0)|$ for $t\le 0$ and thus (assuming without any loss of generality that $K\ge 1$)
\begin{equation}\label{x1}
\sup_{t\in \mathbb R}|x_1(t)| \le NK\sup_{t\ge 0}|z(t)|.
\end{equation}
In addition, 
\begin{equation}\label{ab1}
x_1'(t)=A(t)x_1(t)\ \text{for $t<0$} \ \text{and}  \ x_1'(0-)=A(0)x_1(0),
\end{equation}
where $x_1'(0-)$ denotes the left-derivative of $x_1$ at $0$.

For $t\le 0$, let 
\[
x_2(t)=\int_{-\infty}^tT(t,s)P^-(s)z(s)\, ds-\int_t^0 T(t,s)Q^-(s)z(s)\, ds,
\]
where $Q^-(s)=\Id-P^-(s)$. Then, \eqref{sum} implies that $\sup_{t\le 0}|x_2(t)| \le K\sup_{t< 0}|z(t)|$. Observe that
\begin{equation}\label{ab2}
x_2'(t)=A(t)x_2(t)+z(t) \ \text{for $t<0$} \ \text{and}  \ x_2'(0-)=A(0)x_2(0)+z(0).
\end{equation}
Moreover, $x_2(0)\in \Ima P^-(0)\subset \Ima P^+(0)$. Hence, by setting $x_2(t):=T(t,0)x_2(0)$ for $t>0$ we can extend $x_2$ to a bounded function on $\mathbb R$. Moreover, similarly to~\eqref{x1}, we have that 
\begin{equation}\label{x2}
\sup_{t\in \mathbb R}|x_2(t)| \le NK\sup_{t< 0}|z(t)|.
\end{equation}
In addition, 
\begin{equation}\label{ab3}
x_2'(t)=A(t)x_2(t)\ \text{for $t>0$} \ \text{and}  \ x_2'(0+)=A(0)x_2(0).
\end{equation}
Set $x:=x_1+x_2$. Then,  \eqref{ab}, \eqref{ab1}, \eqref{ab2} and~\eqref{ab3} imply that~\eqref{ADM} holds. Finally, \eqref{x1} and~\eqref{x2} yield that~\eqref{L} holds with $L:=2NK$.
The proof of the lemma is completed.
\end{proof}
Take $\delta>0$ and let  $y\colon \mathbb R\to \mathbb R^d$ be a continuously differentiable function such that~\eqref{pse-c} holds. It follows from the previous lemma that the equation
\[
\tilde x'(t)=A(t)\tilde x(t)+A(t)y(t)-y'(t) \quad t\in \mathbb R
\]
has a solution with the property that $\sup_{t\in \mathbb R}|\tilde x(t)| \le L\delta$. Then, $x:=\tilde x+y$ is  a  solution of Eq.\eqref{LDE} such that $\sup_{t\in \mathbb R}|x(t)-y(t)| \le L\delta$. The proof of the theorem is completed.
\end{proof}
\begin{remark}
In~\cite{BD}, the version of Theorem~\ref{yyy} was established in the particular case when $\sup_{t\in \mathbb R}\|A(t)\|<+\infty$. It turns out that in that case, Hyers-Ulam stability of Eq.~\eqref{LDE} is equivalent to exponential trichotomy (see Remark~\ref{ET}).
\end{remark}

We now have the following version of Theorem~\ref{99} for continuous time.
\begin{theorem}\label{qq}
The following properties are equivalent:
\begin{enumerate}
\item Eq.~\eqref{LDE} is Hyers-Ulam stable with uniqueness;
\item Eq.~\eqref{LDE} admits a summable dichotomy on $\mathbb R$.
\end{enumerate}
\end{theorem}

\begin{proof}
Suppose that Eq.~\eqref{LDE} is Hyers-Ulam stable with uniqueness. Then, one can easily show that $x$ in the statement of Lemma~\ref{09} is unique. By the result stated in~\cite[p.136]{Coppel}, we have that Eq.~\eqref{LDE}  admits a summable dichotomy. The converse can be easily established by using that if Eq.~\eqref{LDE} admits a summable dichotomy, then the only bounded solution of Eq.~\eqref{LDE}  is the trivial one (see again~\cite[p.136]{Coppel}).
\end{proof}

\begin{remark}
We would like to briefly compare the results in this section with those obtained in~\cite{D}. In~\cite{D} the author has studied the problem of characterizing dichotomies with very general growth rates for Eq.~\eqref{LDE} posed on the half-interval $\mathbb R^+$ in terms of certain modified versions of Hyers-Ulam stability.  These results in particular yield (see~\cite[Corollary 2]{D}) characterization of \emph{exponential} dichotomy of Eq.~\eqref{LDE} (on $\mathbb R^+)$ in terms of two types of Hyers-Ulam stability without imposing any bounded growth conditions. 

Our Theorems~\ref{yyy} and~\ref{qq} yield characterizations of \emph{summable} dichotomies and \emph{trichotomies} in terms of a  single and classical concept of Hyers-Ulam stability for Eq.~\eqref{LDE} on $\mathbb R$. We stress that the concepts of trichotomy considered in this paper can only be related to Eq.~\eqref{LDE} posed on $\mathbb R$ (see Definition~\ref{ST} and Remark~\ref{ET}).

Thus, none of the results in~\cite{D} imply the results in this paper and vice-versa.
\end{remark}
\subsection*{Acknowledgments}
The author is grateful to M. Pituk and X. Tang for useful discussions.

\subsection*{Funding}
 D. D. was supported in part by Croatian Science Foundation under the project IP-2019-04-1239 and by the University of Rijeka under the projects uniri-prirod-18-9 and uniri-prprirod-19-16.

\subsection*{Disclosure statement}
The author reports no conflict of interest.

\end{document}